\documentclass[12pt]{article}

\usepackage{graphicx}
\usepackage{psfrag} 
\usepackage{mathptmx}      
\usepackage{amssymb}
\usepackage{amsmath}
\usepackage{amsthm}
\usepackage[dvipsnames]{xcolor}
\usepackage[numbers]{natbib}
\usepackage{marvosym}
\usepackage{a4wide}

\usepackage[mathscr]{euscript}

\theoremstyle{plain}
\newtheorem{theorem}{Theorem}
\newtheorem{lemma}[theorem]{Lemma}
\newtheorem{corollary}[theorem]{Corollary}
\newtheorem{proposition}[theorem]{Proposition}

\theoremstyle{definition}

\newtheorem{conjecture}[theorem]{Conjecture}
\newtheorem{open}[theorem]{Open Problem}

\theoremstyle{remark}
\newtheorem{remark}[theorem]{Remark}

\title{Compatible split systems on a multiset}
\author{Vincent Moulton\\
\small School of Computing Sciences\\[-0.8ex]
\small University of East Anglia\\[-0.8ex] 
\small Norwich, NR4 7TJ, UK\\
\small\tt v.moulton@uea.ac.uk\\
\and
Guillaume E. Scholz\\
\small Bioinformatics Group, Department of Computer Science,\\[-0.8ex]
\small Interdisciplinary Center for Bioinformatics\\[-0.8ex]
\small Universit\"at Leipzig\\[-0.8ex]
\small D-04107 Leipzig, Germany\\
\small\tt guillaume@bioinf.uni-leipzig.de}

\begin{document}

\maketitle

\begin{abstract}
A split system on a multiset $\mathcal M$ is a set 
of bipartitions of $\mathcal M$.  
Such a split system $\mathfrak S$ is 
compatible if it can be represented 
by a tree in such a way that the vertices of the tree
are labelled by the elements in $\mathcal M$, the removal of each edge 
in the tree yields a bipartition in $\mathfrak S$ by taking
the labels of the two resulting components, and every 
bipartition in $\mathfrak S$ can be obtained from the tree in this way.
In this contribution,  we present a novel characterization for compatible 
split systems, and for split systems admitting a unique tree representation. 
In addition, we show that a conjecture on compatibility
stated in 2008 holds for some large classes of split systems.
\end{abstract}

\section{Introduction}\label{sec-prelim}

Let $\mathcal M$ be a multiset with underlying set $X$. For $x \in X$, we denote 
by $\mathcal M(x) \geq 1$ the multiplicity of $x$ in $\mathcal M$, 
and we put $\Delta(\mathcal M)=\sum_{x \in X} (\mathcal M(x)-1)$. 
To ease notation,  we sometimes write $a_1 a_2 \ldots a_n$ for a 
multiset $\{a_1, a_2, \ldots, a_n\}$, and if
an element $a_i$ has multiplicity $k > 1$, then we also denote this by writing $a_i^k$.
We denote by $\mathcal M^* \subseteq X$ the set of 
elements of $X$ with multiplicity $1$ in $\mathcal M$, that is, $\mathcal M^*=\{x \in X, \mathcal M(x)=1\}$. 
Similarly, for $A \subseteq \mathcal M$, we denote by $A^*$ the set of elements of $A$ 
with multiplicity $1$ in $\mathcal M$, that is, $A^*=A \cap \mathcal M^*$. 
Note that throughout this paper, all unions are multiset unions, unless stated otherwise.

A \emph{split} (or \emph{bipartition}) $S$ of $\mathcal M$ is a pair $\{A,B\}$ 
such that $A$, $B$ are nonempty sub(multi)sets of $\mathcal M$, and the 
multiset union $A \cup B$ is precisely $\mathcal M$. We usually 
write $S=A|B$ (or $S=B|A$, as the roles of $A$ and $B$ are symmetric). 
When the set $\mathcal M$ is clear from the context, we also sometimes 
write $A|\overline{A}$ instead of $A|B$, where $\overline A=\mathcal M-\{A\}$.  
A (multi)set $\mathfrak S$ of splits $A|B$ of $\mathcal M$ is 
called a \emph{split system} on $\mathcal M$.

A \emph{labeled tree} is a pair $\mathcal T=(T,\lambda)$ such 
that $T$ is a tree (that is, an undirected, acyclic graph), and $\lambda$ is 
a map from the vertex set $V(T)$ of $T$ to $\mathcal P(\mathcal M)$, the power set of $\mathcal M$. 
An \emph{$\mathcal M$-tree} is a labeled tree $(T,\lambda)$ satisfying the following two properties:
\begin{itemize}
\item[(M1)] The multiset union $\cup_{v \in V(T)} \lambda(v)$ is $\mathcal M$.
\item[(M2)] All vertices $v$ of $T$ of degree $1$ or $2$ satisfy $\lambda(v) \neq \emptyset$.
\end{itemize}
For $v \in V(T)$, we call $\lambda(v)$ the label of $v$, and we say that an 
element $a \in \mathcal M$ (\emph{resp.} a subset $A \subseteq \mathcal M$) \emph{labels} $v$ 
if $a \in \lambda(v)$ (\emph{resp.} $A \subseteq \lambda(v)$). Abusing terminology, 
we also sometimes call a vertex (\emph{resp. } edge) of $T$ a vertex (\emph{resp} edge) 
of $\mathcal T$. For example, the labeled tree depicted in Figure~\ref{fig-intro} is 
an $\mathcal M$-tree for $\mathcal M=\{a^2,b^2,c^2,x,y\}$. We say that two $\mathcal M$-trees $\mathcal T=(T,\lambda)$ and $\mathcal T'=(T',\lambda')$ are \emph{isomorphic} if there exists a graph isomorphism $\phi: V(T) \to V(T')$ such that $\lambda'(\phi(v))=\lambda(v)$ for all $v \in V(T)$.

\begin{figure}[h]
\begin{center}
\includegraphics[scale=1]{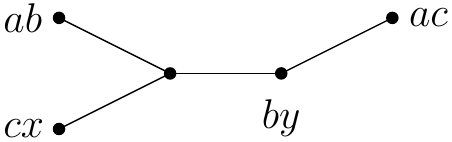}
\caption{A $\mathcal M$-tree representing the split 
system $\{ab|\overline{ab}, ac|\overline{ac}, cx|\overline{cx}, abcx|abcy\}$ on $\mathcal M=\{a^2,b^2,c^2,x,y\}$. 
The labels of each vertex are indicated next to the vertex. 
The multiset union of all labels is $\mathcal M$, so (M1) is 
satisfied. Moreover, all  vertices of degree $1$ or $2$ have a 
nonempty label, so (M2) is satisfied.}
\label{fig-intro}
\end{center}
\end{figure}

Let $\mathcal T=(T,\lambda)$ be a $\mathcal M$-tree. Since $T$ is a tree, the 
removal of an edge $e$ from $T$ results in a graph with exactly two connected 
components $T_1$ and $T_2$. Putting $A=\bigcup_{v \in V(T_1)}\lambda(v)$ 
and $B=\bigcup_{v \in V(T_2)}\lambda(v)$, it follows from (M1) that $A|B$ is a split of $\mathcal M$, 
which we call the split \emph{induced} by $e$. The (multiset) union $\mathfrak S(\mathcal T)$ of 
splits induced by all edges of $T$ is called the split system \emph{represented} by $\mathcal T$.
We say that a split system $\mathfrak S$ on $\mathcal M$
is \emph{compatible} if there exists a $\mathcal M$-tree $\mathcal T=(T,\lambda)$ 
representing $\mathfrak S$, that is, satisfying $\mathfrak S=\mathfrak S(\mathcal T)$. 
For example, the split system $\{ab|\overline{ab}, ac|\overline{ac}, cx|\overline{cx}, abcx|abcy\}$ on $\mathcal M=a^2b^2c^2xy$ 
is compatible, and admits the $\mathcal M$-tree depicted in Figure~\ref{fig-intro} as a representation.

It is not difficult to see that if 
a split system $\mathfrak S$ is compatible, then it
is \emph{pairwise compatible}, that is,  for any pair of splits $S_1, S_2 \in \mathfrak S$ 
there is some $A \in S_1$ and $B \in S_2$ such that $A \cap B = \emptyset$.
In 1971, Buneman proved the following corner-stone result concerning compatibility \cite{B71}:

\begin{theorem}[\cite{B71}]\label{RefBu}
Let $\mathfrak S$ be a set of splits of a set $X$. Then $\mathfrak S$ 
is compatible if and only if $\mathfrak S$ is pairwise compatible.
Moreover,  if $\mathfrak S$ is compatible, then there exists a 
unique (up to isomorphism) $X$-tree $\mathcal T$ representing $\mathfrak S$. 
\end{theorem}

However, as was remarked in \cite{G06},
the equivalence stated in the last theorem does not necessarily hold for 
split systems on multisets. For example, the split system $\mathfrak S=\{ab|acd,ac|abd,ad|abc\}$ 
on $\mathcal M=\{a^2,b,c,d\}$ is pairwise compatible, but it is not compatible. Moreover, 
it was also shown that the uniqueness condition in the theorem can also fail 
to hold for compatible split systems in general.

Despite these issues, in \cite{HLMS08} it was shown that compatibility of a 
split system on a multiset $\mathcal M$ can in fact be characterised by taking into account the quantity $\Delta(\mathcal M)$:

\begin{theorem}[\cite{HLMS08}, Thm. 4.3]\label{ref43}
Let $\mathfrak S$ be a split system on a multiset $\mathcal M$. Then $\mathfrak S$ is 
compatible if and only if every submultiset of $\mathfrak S$ 
of size at most $\mathrm{max}\{2\Delta(\mathcal M),\Delta(\mathcal M)+2\}$ is compatible.
\end{theorem}

In addition, the authors of \cite{HLMS08}  stated the following conjecture:

\begin{conjecture}\label{mainconj}
Let $\mathfrak S$ be a split system on a multiset $\mathcal M$. Then, 
\item[($\bigstar$)] $\mathfrak S$ is compatible if and only if every submultiset of $\mathfrak S$ of 
size at most $\Delta(\mathcal M)+2$ is compatible.
\end{conjecture}

This contribution is organized as follows. In Section~\ref{sec-scg}, we 
introduce the split-containment graph $\Gamma(\mathfrak S)$ of a split system $\mathfrak S$. 
We then use that graph to state a characterization of compatibility of a split system
for the multiset case (Theorem~\ref{char-i}). In addition, we show 
that $\Gamma(\mathfrak S)$ can be used to count the number of non-isomorphic tree representations of a compatible split system. This gives 
rise to a characterization of compatible split systems admitting a unique representation (Corollary~\ref{corun}).

In Section~\ref{sec-supt} and \ref{sec-23}, we turn our attention to 
Conjecture~\ref{mainconj}. In particular, we show that for a split system $\mathfrak S$ on some multiset $\mathcal M$, 
the equivalence ($\bigstar$) in Conjecture~\ref{mainconj} holds in case  
(1) $\mathfrak S$ contains a certain number of splits that satisfying a certain 
set-inclusion property (Theorem~\ref{tmd2}), 
(2) the graph $\Gamma(\mathfrak S)$ enjoys a 
sparsity property (Corollary~\ref{cor-thin}), and 
(3) all of the splits in $\mathfrak S$ have the same size, where 
the size of a split $A|B$ equals $\min\{|A|,|B|\}$ (Theorem~\ref{equal}). 
Finally, using (1) we show that in case all of the splits in 
a split system $\mathfrak S$ have size at most 3, 
then $\mathfrak S$ satisfies a slightly weaker version of $(\bigstar)$, in which 
$\Delta(\mathcal M)+2$ is replaced by $\Delta(\mathcal M)+3$  (Theorem~\ref{twothreeholds}).

Before proceeding, we remark that 
in case $\mathcal M$ is a set, 
$\mathcal M$-trees and their relationship with 
compatible split systems form a fundamental part of the 
underlying theory for the area of phylogenetics (see e.g. \cite[Chapter 3]{SS03}).
Moreover, $\mathcal M$-trees for $\mathcal M$ a multiset are closely
related to  \emph{multi-labeled phylogenetic trees} (or \emph{MUL-trees} for short), 
that arise in the context of polyploid studies and phylogenetic network theory (see e.g. \cite{huber2006phylogenetic,PO01}). 
Roughly speaking, MUL-trees are rooted trees equipped with a (not necessary injective) 
function from their leaf set to some multiset. Other 
related structures are \emph{tangled trees} \cite{P94}, introduced 
as a way to model host-parasite co-evolution, 
and \emph{area cladograms} \cite{G06}, that are used in biogeographical studies.


\section{The Split-Containment Graph}\label{sec-scg}

As we have seen in the introduction, pairwise compatibility of
a split system does not necessarily imply compatibility.
However, in this section we show that we can characterize
compatibilty in terms of a certain graph that we shall associate
to a split system. As we shall see, this graph also yields a characterization  
for when compatible split system is represented by a unique tree. 

We begin by defining the graph.
For a split system $\mathfrak S=\{S_1, \ldots, S_n\}$, $n \geq 1$ we define 
the \emph{split-containment graph} of $\mathfrak S$ $\Gamma(\mathfrak S)$ of $\mathfrak S$ as follows. 
The vertex set of $\Gamma(\mathfrak S)$ is the multiset $\{(A, S_i): A \in S_i, 1 \leq i \leq n\}$, 
and the arc set of $\Gamma(\mathfrak S)$ is the multiset of 
ordered pairs $((A, S_i), (B, S_j))$ satisfying $A \subsetneq B$ (as multisets) 
and $i \neq j$. Note that since $\mathcal S$ is a multiset, $S_i=S_j$ may hold.

Clearly the graph $\Gamma(\mathfrak S)$ is acyclic. Moreover, that graph satisfies the property that if $((A,S_i),(B,S_j))$ 
is an arc of $\Gamma(\mathfrak S)$, then $((\overline B,S_j),(\overline A,S_i))$ is an arc of $\Gamma(\mathfrak S)$.
It follows that for any $i,j$ distinct, $\Gamma(\{S_i,S_j\})$ either 
contains no edges or it is isomorphic to one of the digraphs $D_1$ or $D_2$, 
where $D_1$ and $D_2$ are digraphs on four vertices $\{v, w, p, q\}$ such that 
$D_1$ has arcs $(v, p), (q, w)$ and $D_2$ has arcs $(v, p), (p, w), (v, q), (q, w)$ (see Figure~\ref{fig-dig}).

\begin{figure}[h]
	\begin{center}
		\includegraphics[scale=1]{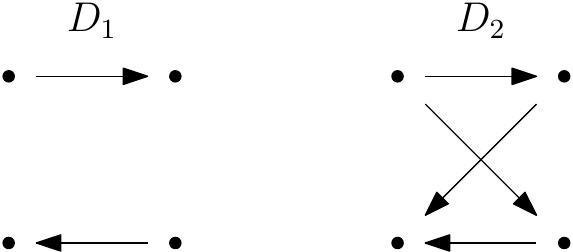}
		\caption{For $\mathfrak S=\{S_1, \ldots, S_n\}$ a split system and $i,j  \in \{1, \ldots, n\}$ distinct, either $\Gamma(\{S_i,S_j\})$ has no edges, or $\Gamma(\{S_i,S_j\})$ is isomorphic to one of the depicted graphs $D_1$ and $D_2$.}
		\label{fig-dig}
	\end{center}
\end{figure}

In view of these observations, we obtain the following lemma: 

\begin{lemma}
	Suppose $\mathfrak S = \{S_1, S_2\}$ is a split system on $\mathcal M$. 
	Then the following are equivalent:
	\begin{itemize}
	\item[(i)] $\mathfrak S$ is compatible.
	\item[(ii)] $\Gamma(\mathfrak S)$ is isomorphic to $D_1$ or $D_2$.
	\item[(iii)] The edge set of $\Gamma(\mathfrak S)$ is non-empty.
	\end{itemize}
	Moreover, if $\mathfrak S$ is compatible, then there is a unique $\mathcal M$-tree 
	representing $\mathfrak S$ if and only if $\Gamma(\mathfrak S)$ is isomorphic to $D_1$.
\end{lemma}

In light of the last lemma, we call a subgraph $G$ of $\Gamma(\mathfrak S)$ \emph{thin} if $V(G)=V(\Gamma(\mathfrak S))$ and 
the restriction of $G$ to $\{(A, S_i), (\overline A, S_i), (B, S_j), (\overline B, S_j)\}$ 
is isomorphic to $D_1$ for all $i,j \in \{1, \ldots, n\}$ distinct. 
We say that $\mathfrak S$ is \emph{thin} if $\Gamma(\mathfrak S)$ is a 
thin subgraph of itself. In particular, if $\mathcal M$ is a set, 
then a split system $\mathfrak S$ on $\mathcal M$ is thin 
if and only if $\mathfrak S$ is compatible. Note that there 
exist split systems $\mathfrak S$ on multisets that are thin but not compatible,
and compatible but not thin. 
For example, the split system $\{xx|xyz,xy|xxz,xz|xxy\}$ 
on $\mathcal M=\{x^3,y,z\}$ is thin but not compatible, and the 
the split system  $\{ab|\overline{ab}, ac|\overline{ac}, cx|\overline{cx}, abcx|abcy\}$ on 
$\mathcal M=\{a^2,b^2,c^2,x,y\}$ represented by the $\mathcal M$-tree depicted 
in Figure~\ref{fig-intro} is not thin.

We now consider the structure of thin subgraphs of $\Gamma(\mathfrak S)$ in 
more detail. For $G$ a thin subgraph of $\Gamma(\mathfrak S)$, we call an 
arc $((A, S_i), (B, S_j))$ of $G$ \emph{critical} if 
there does not exist a directed path from $(A, S_i)$ to $(B, S_j)$ in $G$ 
other that the path formed by the single arc $((A, S_i), (B, S_j))$. 
Note that if $((A, S_i), (B, S_j))$ is a critical arc of $G$, then the 
corresponding arc $((\overline B, S_j), (\overline A, S_i))$ is also a critical 
arc of $G$. 
In addition, we call a thin subgraph $G$ of $\mathfrak S$ \emph{consistent} if
\[
\bigcup_{(B, S_j) \in \mathcal C_G((A,S_i))} B \subseteq A.
\]
holds for all $(A,S_i) \in V(G)$, where $\mathcal C_G((A,S_i))$ is the set of all vertices $(B,S_j)$ of $G$ such that $((B,S_j),(A,S_i))$ is a critical arc of $G$.

Interestingly, consistent thin subgraphs are uniquely determined by their critical arcs. Indeed, we have:

\begin{lemma}\label{lm-deter}
Let $\mathfrak S=\{S_1, \ldots, S_n\}$, $n \geq 1$, be a split system 
on a multiset $\mathcal M$, and let $G$ be a consistent thin subgraph of $\Gamma(\mathfrak S)$. 
Then two vertices $(A,S_i), (B,S_j)$, $i \neq j$, $i,j \in \{1,\dots, n\}$, of $G$ are 
joined by an arc if and only if there is a path in $G$ from $(A,S_i)$ to $(B,S_j)$ that contains only critical arcs.
\end{lemma}

\begin{proof}
One direction is trivial: If $((A,S_i),(B,S_j))$ is an arc of $G$, then 
either $((A,S_i),(B,S_j))$ is a critical arc of $G$, or there is a path in 
$G$ from $(A,S_i)$ to $(B,S_j)$ of length two or more, such that all arcs on this path are critical arcs.

Conversely, suppose that there exists a path in $G$ from $(A,S_i)$ to $(B,S_j)$
that contains only critical arcs. In particular, $A \subsetneq B$ holds. Since $G$ 
is thin, exactly one of  the ordered 
pairs $((A,S_i),(B,S_j))$, $((B,S_j),(A,S_i))$, $((\overline A,S_i),(B,S_j))$, 
and $(B,S_j)), (\overline A,S_i))$ must be an arc of $G$. So it suffices to show that 
for the last three pairs this is impossible.

If $((B,S_j),(A,S_i))$ is an arc of $G$, then $B \subseteq A$ 
holds, which is impossible given that $A \subsetneq B$ holds.

If $((\overline A,S_i),(B,S_j))$ is an arc of $G$, then 
there exists a path $P_1$ from $(\overline A,S_i)$ to $(B,S_j)$ in $G$ 
that contains only critical arcs. By assumption, 
there also exists a path $P_2$ in $G$ from $(A,S_i)$ to $(B,S_j)$ 
that contains only critical arcs. Now, let $(C,S_k)$ be the first vertex 
that is common to $P_1$ and $P_2$. Such a vertex must exist, 
since $P_1$ and $P_2$ have the same end-vertex $(B, S_j)$. 
Since $G$ is consistent, and the paths from $(\overline A,S_i)$ 
to $(C,S_k)$ and from $(A,S_i)$ to $(C,S_k)$ are vertex-disjoint 
by choice of $(C,S_k)$, it follows that $\mathcal M=\overline A \cup A \subseteq C$, 
which is impossible.

Finally, if $((B,S_j),(\overline A,S_i))$ is an arc of $G$, then 
there exists a path $P_1$ from $(B,S_j)$ to $(\overline A,S_i)$ in $G$ 
that contains only critical arcs. Since by assumption, there 
exists a path in $G$ from $(A,S_i)$ to $(B,S_j)$ that contains only critical arcs, 
there must exist a path $P_2$ in $G$ from $(\overline B,S_j)$ to $(\overline A,S_i)$ 
with the same property. Now, let $(C,S_k)$ be the first vertex that is 
common to $P_1$ and $P_2$. Such a vertex must exist, since $P_1$ and $P_2$ 
have the same end-vertex $(\overline A, S_i)$. Since $G$ is consistent, and 
the paths from $(B,S_j)$ to $(C,S_k)$ and from $(\overline B,S_j)$ to $(C,S_k)$ 
are vertex-disjoint by choice of $(C,S_k)$, it follows that 
$\mathcal M=B \cup \overline B \subseteq C$, which is impossible.
\end{proof}

In addition, the existence of certain pairs of critical arcs in a 
consistent thin subgraph forces the existence of further critical arcs:

\begin{lemma}\label{lm-trans}
	Let $\mathfrak S=\{S_1, \ldots, S_n\}$, $n \geq 1$, be a 
	split system on a multiset $\mathcal M$, and let $G$ be a consistent 
	thin subgraph of $\Gamma(\mathfrak S)$. If $i,j,k \in \{1, \ldots, n\}$ 
	are such that $((\overline A,S_i),(B,S_j))$ and $((\overline C, S_k),(B,S_j))$ 
	are critical arcs of $G$, then $((\overline A,S_i),(C,S_k))$ is a critical arc of $G$.
\end{lemma}

\begin{figure}[h]
	\begin{center}
		\includegraphics[scale=1]{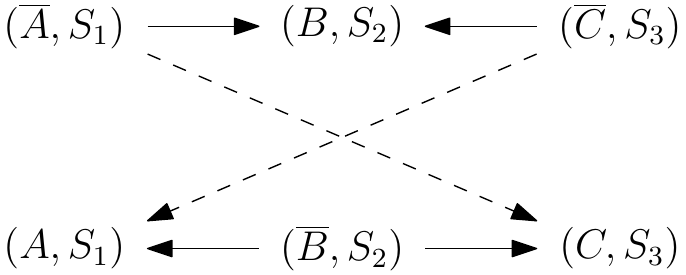}
		\caption{A thin subgraph $G$ of $\Gamma(\mathfrak S)$, 
			where $\mathfrak S=\{A|\overline A, B|\overline B, C|\overline C\}$. 
			By Lemma~\ref{lm-trans}, if $G$ is consistent, and the four 
			solid arcs are critical arcs of $G$, then 
			the dashed arcs must be critical arcs of $G$.}
		\label{help-tr}
	\end{center}
\end{figure}

\begin{proof}
	We first show that $((\overline A,S_i),(C,S_k))$ is an arc of $G$. 
	Since $G$ is thin, exactly one of the ordered pairs 
	$((\overline A,S_i),(C,S_k))$, $((\overline A,S_i),(\overline C,S_k))$, 
	$((C,S_k),(\overline A,S_i))$ and $((\overline C,S_k),(\overline A,S_i))$ 
	must be an arc of $G$. Clearly, $((\overline A,S_i),(\overline C,S_k))$ is not 
	an arc of $G$, as in this case $((\overline A,S_i),(B,S_j))$ is 
	not a critical arc of $G$. By symmetry, $((\overline C,S_k),(\overline A,S_i)$ 
	is not an arc of $G$ either. It is also impossible to have $((C,S_k),(\overline A,S_i))$ 
	an arc in $G$. Indeed, $\mathfrak S$ is consistent, so
	$B$ contains the multiset union $\overline A \cup \overline C$. If $C \subset \overline A$ holds, then it follows that $\mathcal M= C \cup \overline C \subseteq A \cup \overline C \subseteq B$, which is impossible.
 Hence, $((\overline A,S_i),(C,S_k))$ is an arc of $G$.
	
	Now, assume for contradiction that the arc $((\overline A,S_i),(C,S_k))$ is not critical. 
	This means that there exists a further vertex $(D,S_l)$ of $G$ such that $G$ contains 
	a directed path from $(\overline A,S_i)$ to $(D,S_l)$, and a directed path from $(D,S_l)$ to $(C,S_k)$. 
	Without loss of generality, we can choose $(D,S_l)$ such that $((D,S_l),(C,S_k))$ 
	is a critical arc of $G$. Since $G$ is thin, exactly one of $((B,S_j),(D,S_l))$, $((D,S_l),(\overline B,S_j))$, $((\overline B,S_j),(D,S_l))$ and $((D,S_l),(B,S_j))$ must be an arc of $G$. 
	We next show that neither of these arcs can be an arc of $G$.
	
	Since $G$ is consistent, $C$ contains the multiset union $D \cup \overline B$. Hence, $B \subseteq D$ 
	cannot hold, as this would imply $\mathcal M \subseteq C$, so $((B,S_j),(D,S_l))$ is not an arc of $G$.
	
     If $((D,S_l),(\overline B,S_j))$ is an arc of $G$, then this contradicts the assumption 
     that $((D,S_l),(C,S_k))$ is critical. Similarly, if $((\overline B,S_j),(D,S_l))$ is an arc of $G$, then 
	this contradicts the assumption that $((\overline B,S_j),(C,S_k))$ is critical. So, 
	neither $((D,S_l),(\overline B,S_j))$ nor $((\overline B,S_j),(C,S_k))$ are arcs of $G$.
	
	Finally, suppose $((D,S_l),(B,S_j))$ is an arc of $G$. By assumption, there is a directed path in $G$ from
	$(\overline A, S_i)$ to $(D,S_l)$. Since $\Gamma(\mathfrak S)$ is acyclic, this path does not contain
	the arc $((\overline A, S_i),(B,S_j))$.  But this contradicts the assumption 
	that $((\overline A, S_i),(B,S_j))$ is critical. Hence, $((D,S_l),(B,S_j))$ is not an arc of $G$.
	
	In summary, the existence of a vertex $(D,S_l)$ distinct from $(\overline A,S_i)$ and $(C,S_k)$ such that $G$ contains 
	a directed path from $(\overline A,S_i)$ to $(D,S_l)$, and a directed path from $(D,S_l)$ to $(C,S_k)$ is impossible. Hence, $((\overline A,S_i),(C,S_k))$ is a critical arc of $G$.
\end{proof}

We are now in a position to give a characterization for when a split system $\mathfrak S$ is compatible in terms of the existence of consistent thin subsets of $\Gamma(\mathfrak S)$.

\begin{theorem}\label{char-i}
	Let $\mathfrak S$, $n \geq 1$, be a split system on a multiset $\mathcal M$. 
	Then $\mathfrak S$ is compatible if and only if there exists a consistent thin subgraph $G$ of $\Gamma(\mathfrak S)$.
\end{theorem}

\begin{proof}
	Let $\mathfrak S=\{S_1, \ldots S_n\}$, $n \geq 1$.
	Assume first that $\mathfrak S$ is compatible. Let $\mathcal T=(T,\lambda)$ 
	be a representation of $\mathfrak S$. By definition, 
	there exists a bijection $e: \mathfrak S \to E(T)$ such that for all $i \in \{1, \ldots n\}$, 
	the graph $T-e(S_i)$ has two connected components, each labeled with one part of $S_i$. 
	We construct $G$ from $\mathcal T$ as follows: 
	The vertices of $G$ are the vertices of $\Gamma(\mathfrak S)$. The arcs of $G$ 
	are the arcs $((A,S_i),(B,S_j))$ such that the connected component of $T-e(S_i)$ 
	labeled with $A$ is a subgraph of the connected component of $\mathcal T-e(S_j)$ labeled with $B$. 
	As $\mathcal T$ is an $\mathcal M$-tree
	this defnition implies that $A \subsetneq B$ for all arcs $((A,S_i),(B,S_j))$ of $G$, so $G$ is a subgraph of $\Gamma(\mathfrak S)$. 
	Moreover, for any two $i,j \in \{1, \ldots, n\}$ distinct, the subgraph of $G$ induced 
	by $\{(A, S_i), (\overline A, S_i), (B, S_j), (\overline B, S_j)\}$ contains exactly two arcs, 
	so $G$ is thin. To see that $G$ is consistent, let $(A,S_i)$ be a vertex of $G$. By definition, an arc $((C,S_k),(A,S_i))$ of $G$ 
	is critical if and only if the edges $e(S_i)$ and $e(S_j)$ share a vertex $v$ in $T$. 
	It then follows directly that $\bigcup_{(C,S_k) \in \mathcal C_G((A,S_i))} C=A-\lambda(v) \subseteq A$, which 
	proves that $G$ is consistent.
	
	Conversely, assume that $\Gamma(\mathfrak S)$ has a consistent thin subgraph $G$. 
	We define $\lambda_G:V(G) \to \mathcal P(\mathcal M)$ by putting, 
	for $(A,S_i) \in V(G)$, $\lambda(A,S_i)=A-\bigcup_{(C,S_k) \in \mathcal C_G((A,S_i))} C$. 
	Since $G$ is consistent, the union is contained in $A$, so $\lambda_G$ is well defined.
	
	Next, we define an equivalence relation $\sim_G$ on $V(G)$ as follows: for $(A,S_i),(B,S_j) \in V(G)$, 
	we put $(A,S_i) \sim_G (B,S_j)$ if and only if $(A,S_i)$ and $(B,S_j)$ are the same vertex of $G$, or $((\overline A,S_i),(B,S_j))$ is 
	a critical arc of $G$. By definition, $\sim_G$ is reflexive. It is also symmetric, 
	as $((\overline A,S_i),(B,S_j))$ is critical if and only if $((\overline B,S_j),(A,S_i))$ is critical. 
	Finally, transitivity is a direct consequence of Lemma~\ref{lm-trans}.
	
	We claim that for all pairs $(A,S_i),(B,S_j)$ of vertices of $G$ with $(A,S_i)\sim_G (B,S_j)$, we 
	have $\lambda_G(B,S_j)=\lambda_G(A,S_i)$. To see that, note 
		that by Lemma~\ref{lm-trans}, if $((C,S_k),(B,S_j))$ is a critical arc of $G$ 
		distinct from $((\overline A,S_i),(B,S_j))$, then $((C,S_k),(A,S_i))$ is a critical arc of $G$. 
		Since the roles of $A$ and $B$ are symmetric in this argument it follows that the 
		sets
		$$
		\{(C,S_k) \in V(G)-\{(\overline A,S_i)\}: (C,S_k) \in \mathcal C_G((B,S_j))\}
		$$ 
		and $$
		\{(C,S_k) \in V(G)-\{(\overline B,S_j)\}: (C,S_k) \in \mathcal C_G((A,S_i))\}$$ 
		are equal. Denoting this set by $\mathcal C$, we have $\lambda_G((A,S_i))=(A-\overline B)-\bigcup_{(C, S_k) \in \mathcal C} C$ 
		and $\lambda_G((B,S_j))=(B-\overline A)-\bigcup_{(C, S_k) \in \mathcal C} C$. 
		Since $A-\overline B=B-\overline A$, it follows that $\lambda_G((A,S_i))=\lambda_G((B,S_j))$ as claimed.

	Now, we denote by $T$ the undirected graph whose vertex set is the set 
	equivalence classes of $\sim_G$, where two equivqlence classes $u$, $v$ are joined by an edge 
	if and only if there exists $i \in \{1, \ldots, n\}$ such that $(A,S_i) \in u$ and $(\overline A, S_i) \in v$.
	Note that by construction, the degree of a vertex $u$ of $T$ is precisely the size of the equivalence class $u$.
	In view of the above, $\lambda_G$ trivially induces a map $\lambda:V(T) \to \mathcal P(\mathcal M)$.

	We next show that $\mathcal T(G)=(T,\lambda)$ is a representation of $\mathfrak S$.
	We do this by induction on $n=|\mathfrak S|$. If $n=1$, this is trivial, as $T$ is a 
	single edge $\{u,v\}$ with $\lambda(u)=A$ and $\lambda(v)=\overline A$, where $A|\overline A$ 
	is the unique element of $\mathfrak S$. Assume then that $n \geq 2$, and that the property holds 
	for all split systems $\mathfrak S'$ with $|\mathfrak S'| < |\mathfrak S|$.
	
	Let $(A,S_i)$ be a vertex of indegree $0$ in $G$ (which exists as $\Gamma(\mathfrak S)$ is acyclic), and let $G'$ be the graph 
	obtained from $G$ by removing the vertices $(A,S_i)$ and $(\overline A,S_i)$ and their adjacent edges. 
	Clearly, $G'$ is a consistent thin subgraph of $\Gamma(\mathfrak S-\{S_i\})$. 
	By our induction hypothesis $\mathcal T(G')=(T',\lambda')$ is a representation of $\mathfrak S-\{S_i\}$. 
	Since $(A,S_i)$ has indegree 0 in $G$, $(A,S_i)$ is the unique element of 
	its equivalence class under $\sim_G$. In particular, there exists a leaf $u$ of $T$ 
	such that $\lambda(u)=A$. Moreover, $(A,S_i)$ has outdegree at least 1 in $G$, 
	so there exists a vertex $(B,S_j)$ of $G$ such that $((A,S_i),(B,S_j))$ is a critical 
	arc of $G$. So, $(\overline A, S_i) \sim_G (B,S_j)$ holds. Finally, since $(A,S_i)$ 
	has indegree $0$ in $G$ (and by symmetry, $(\overline A, S_i)$ has outdegree $0$ in $G$), 
	it follows that an arc $((C,S_k),(B,S_j))$ of $G'$ is a critical arc of $G'$ if and only 
	if $((C,S_k),(B,S_j))$ is a critical arc of $G$. Hence, the equivalence classes 
	of $\sim_G$ are exactly the equivalence classes of $\sim_{G'}$, 
	plus the equivalence class of $(A,S_i)$. In particular, $V(T)=V(T') \cup \{u\}$, 
	and $T$ is obtained from $T'$ by adding leaf $u$ to the vertex $w$ of $T'$ 
	corresponding to the equivalence class of $(\overline A,S_i)$ under $\sim_G$. 
	Moreover, we have $\lambda(v)=\lambda'(v)$ for all $v \in V(T)$ distinct from $u,w$, 
	and, as already stated, $\lambda(u)=A$. Finally, we 
	have $\lambda(w)=B-\bigcup_{(C,S_k) \in \mathcal C_G((B,S_j))} C$ 
	and $\lambda'(w)=B-\bigcup_{(C,S_k) \in \mathcal C_{G'}((B,S_j))} C$, 
	so $\lambda(w)=\lambda'(w)-A$. In particular, $A \subseteq \lambda'(w)$ holds.

Putting these observations together, it follows that $\mathcal T(G)$ is 
obtained from $\mathcal T(G')$ by (i) attaching a new leaf labeled with $A$ to 
	the vertex $w$, and (ii) removing $A$ from 
	the label of $w$. Since $\mathcal T(G')$ is a $\mathcal M$-tree by 
	our induction hypothesis, it follows that $\mathcal T(G)$ is a $\mathcal M$-tree 
	if and only if $w$ is not a leaf of $\mathcal T(G')$ with $\lambda'(w)=A$. 
	To see that this is the case, we remark that the degree of $w$ in $\mathcal T(G')$ 
	is precisely the size of the equivalence class of $\sim_{G'}$ corresponding to $w$. 
	In particular, if $w$ is a leaf, then $(B,S_j)$ is the unique element of its 
	equivalence class under $\sim_{G'}$, and $\lambda'(w)=B$. 
	Since $((A,S_i),(B,S_j))$ is an arc of $\Gamma(\mathfrak S)$, we 
	have $A \subsetneq B$, so $\lambda_{G'}(w) \neq A$ as claimed.	
	
So, $\mathcal T(G)$ is a $\mathcal M$-tree, and we have by construction $\mathfrak S(\mathcal T(G))=\mathfrak S(\mathcal T(G')) \cup \{S_i\}=\mathfrak S$. So, $\mathcal T(G)$ is a representation of $\mathfrak S$.
\end{proof}

If $\mathfrak S$ is a compatible split system and $G$ is a 
consistent thin subgraph of $\Gamma(\mathfrak S)$, we denote by $\mathcal T(G)$ the $\mathcal M$-tree
constructed from $G$ as in the proof of Theorem~\ref{char-i}. As shown in that proof, 
$\mathcal T(G)$ is a representation of $\mathfrak S$. Note that there are 
compatible split systems $\mathfrak S$ for
which $\Gamma(\mathfrak S)$ has two or more distinct consistent
thin subgraphs $G, G'$ such that $\mathcal T(G)$ is isomorphic to 
$\mathcal T(G')$. For example, for the split system $\mathfrak S = \{a|abbcd,abc|abd,b|aabcd\}$ 
on $\mathcal M=\{a^2,b^2,c,d\}$, the graph $\Gamma(\mathfrak S)$ has 
four distinct consistent thin subgraphs, two of which we depict in Figure~\ref{fig-consth}.

\begin{figure}[h]
	\begin{center}
		\includegraphics[scale=.8]{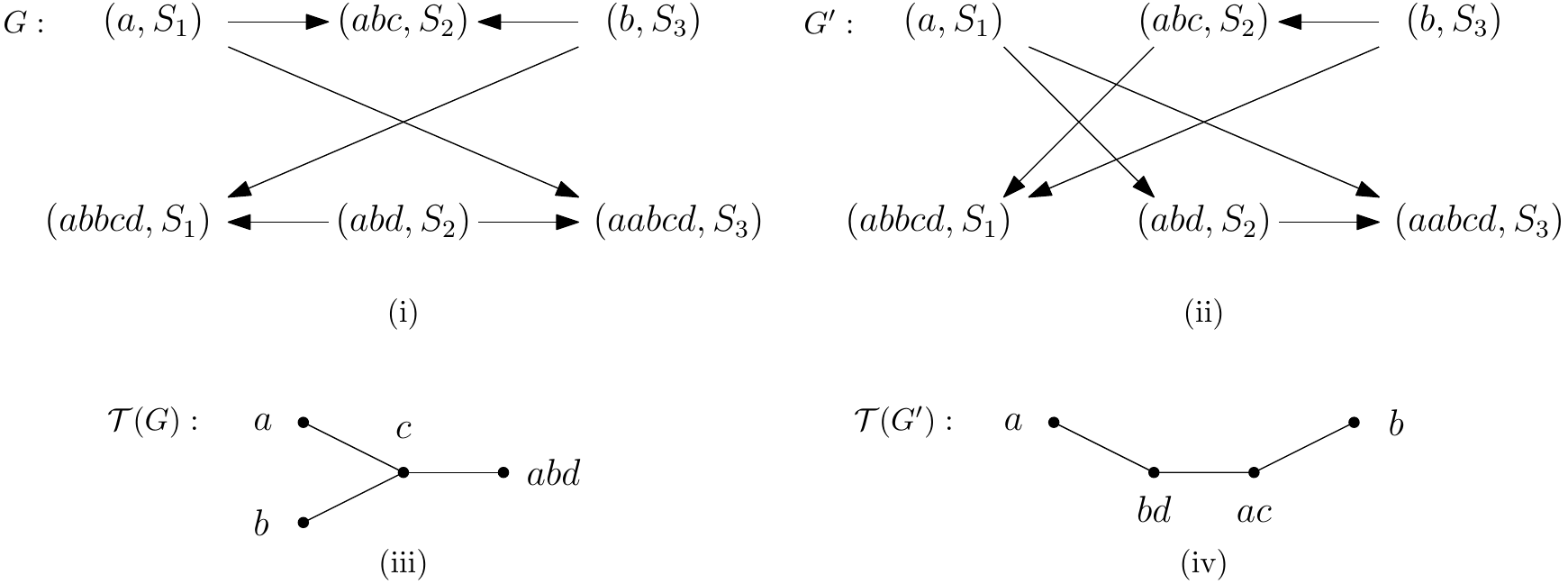}
		\caption{(i) and (ii) Two consistent thin subgraphs of $\Gamma(\mathfrak S)$, where $\mathfrak S=\{a|abbcd,abc|abd,b|aabcd\}$. The other two consistent thin subgraphs of $\Gamma(\mathfrak S)$ can be obtained by swapping the vertices $(abc,S_2)$ and $(abd,S_2)$ in $G$ and in $G'$. (iii) and (iv) The $\mathcal M$-trees $\mathcal T(G)$ and $\mathcal T(G')$ constructed from $G$ and $G'$ respectively, as in the proof of Theorem~\ref{char-i}. Both $\mathcal T(G)$ and $\mathcal T(G')$ are representations of $\mathfrak S$. The other two representations of $\mathfrak S$ can be obtained by swapping labels $c$ and $d$ in $\mathcal T(G)$ and in $\mathcal T(G')$.}
		\label{fig-consth}
	\end{center}
\end{figure}

The last example motivates the final result in this section.
Suppose that $\mathfrak S$ is a compatible split system. We 
call two distinct thin subgraphs $G$ and $G'$ 
of $\mathfrak S$ {\em isomorphic} if there is a 
graph theoretical digraph isomorphism between
$G$ and $G'$ which maps each vertex of the 
form $(A,S_i)$ to one of the form $(B,S_j)$, where $A=B$.

\begin{corollary}\label{corun}
	Let $\mathfrak S$, be a compatible split system on a multiset $\mathcal M$. Then the
	non-isomorphic representations $\mathcal T$ of $\mathfrak S$ are in bijective
	correspondence with the	non-isomorphic consistent thin subgraphs of $\Gamma(\mathfrak S)$.
	In particular, $\mathfrak S$ has a unique representation if and only if $\Gamma(\mathfrak S)$ has a unique consistent thin subgraph up to isomorphism.
\end{corollary}

\begin{proof}
	Let $\mathfrak S=\{S_1, \ldots, S_n\}$, $n \geq 1$.
	As seen in the proof of Theorem~\ref{char-i}, a 
	representation $\mathcal T$ of $\mathcal S$ trivially induces a 
	consistent thin subgraph $G$ of $\Gamma(\mathfrak S)$. It
	 is straightforward to see that $\mathcal T=\mathcal T(G)$ in that case.

	Now, suppose that $G_1$ and $G_2$ are two distinct 
	consistent thin subgraphs of $\Gamma(\mathfrak S)$. If $G_1$ 
	and $G_2$ are isomorphic, then it is straight-forward to check 
	that $\mathcal T(G_1)$ is isomorphic to $\mathcal T(G_2)$. 
	
	So, suppose $G_1$ and $G_2$ are not isomorphic.
	Since by Lemma~\ref{lm-deter}, $G_1$ and $G_2$ are 
	uniquely determined by their sets of critical arcs, 
	there exists a vertex $(A,S_i)$ of $\Gamma(\mathfrak S)$ 
	such that (i) the sets $\mathcal C_{G_1}((A,S_i))$ and $\mathcal C_{G_2}((A,S_i))$ 
	are distinct, and (ii) no further vertex $(B,S_j)$ of $\Gamma(\mathfrak S)$ 
	satisfying $B=A$ and $\mathcal C_{G_2}((B,S_j))=\mathcal C_{G_1}((A,S_i))$. 
	Denoting by $v_{G_k}$, $k \in \{1,2\}$, the vertex of $\mathcal T(G_k)$  
	corresponding to the equivalence class of $(A,S_k)$ under $\sim_{G}$, 
	it follows from (i) that the set of splits induced by edges adjacent to 
	$v_{G_1}$ in $\mathcal T(G_1)$ and the set of splits induced by 
	edges adjacent to $v_{G_2}$ in $\mathcal T(G_2)$ are distinct. 
	Moreover, (ii) implies that there is no further vertex $w$ of $\mathcal T(G_2)$ 
	such that the set of splits induced by edges adjacent to $w$ in $\mathcal T(G_2)$ 
	is precisely the set of splits induced by edges adjacent to $v_{G_1}$ 
	in $\mathcal T(G_1)$. Thus, $\mathcal T(G_1)$ and $\mathcal T(G_2)$ are not isomorphic.

	
	The rest of the proof follows from the observation 
	that $G_1$ (\emph{resp.} $G_2$) is precisely the graph 
	obtained from $\mathcal T(G_1)$ (\emph{resp.} $\mathcal T(G_2)$) using the 
	construction described in the first part of the proof of Theorem~\ref{char-i}.
\end{proof}

Note that by Corollary~\ref{corun}, if $\Gamma(\mathfrak S)$ 
is a thin and consistent subgraph of itself, then $\mathfrak S$ has a unique representation. 
However, there exists non-thin split systems $\mathfrak S$
such that $\Gamma(\mathfrak S)$ has a unique (up to isomorphism) thin 
and consistent subgraph, and so it is not necessary for $\mathfrak S$ to 
be thin in order for $\mathfrak S$ to have a unique representation. 
For example, the split system $\mathcal S$ on $\mathcal M=\{a^2,b^2,c^2,x,y\}$ 
represented by the $\mathcal M$-tree $\mathcal T$ depicted 
in Figure~\ref{fig-intro} is not thin. However, $\Gamma(\mathfrak S)$ 
has only one consistent thin subgraph, so $\mathcal T$ is the
unique representation of $\mathfrak S$.


\section{Superterminal sets}\label{sec-supt}

We now focus on vertices of in-degree $0$ of the split-containment 
graph $\Gamma(\mathfrak S)$ of a split system $\mathfrak S$. In particular, 
we show in Theorem~\ref{tmd2} that if a carefully chosen subset of such 
vertices enjoys the extra property that their 
out-degree in $\Gamma(\mathfrak S)$ is precisely $|\mathfrak S|-1$, 
then $\mathfrak S$ satisfies the equivalence ($\bigstar$) in Conjecture~\ref{mainconj}.

We begin with some preliminary observations.

\begin{lemma}\label{lmcol}
Let $\mathfrak S$ be a split system on $\mathcal M$ of size $n \geq 3$ 
such that all proper subsets of $\mathfrak S$ are compatible. If there is a 
split $S=A|\overline A \in \mathfrak S$ such that $A \subseteq \mathcal M^*$, 
then $\mathfrak S$ is compatible.
\end{lemma}

\begin{proof}
Assume that there exists a split $S=A|\overline A$ in $\mathfrak S$ such 
that $A \subseteq \mathcal M^*$. We can choose $S$ in $\mathfrak S$ such that $A$ 
is minimal for inclusion. Note that $S$ must have multiplicity one in $\mathfrak S$, as the 
split system $\{S,S\}$ is not compatible. Thus, because $\mathfrak S-\{S'\}$ is 
compatible for all $S'$ distinct from $S$, it follows that for all $B| \overline B \in \mathfrak S$, 
precisely one of $A \subseteq B$, $A \subseteq \overline B$ holds.

It follows that in any representation $\mathcal T=(T,\lambda)$ of $\mathfrak S-\{S\}$, 
there exists a vertex $v$ of $T$ with $A \subseteq \lambda(v)$. Otherwise, there exists an 
edge $e$ in $T$ and two elements $x,y \in A$ distinct such that the split $S_e=B|\overline B$ of $\mathfrak S$ 
associated to $e$ satisfies (up to permutation) $x \in B$ and $y \in \overline B$. Since $x$ and $y$ 
have multiplicity one in $\mathcal M$, neither $A \subseteq B$ nor $A \subseteq \overline B$ holds, a contradiction.

Thus we can take $A$ out of $v$, by adding a vertex $u$ and edge $\{u,v\}$ to $T$, labelling $u$ with $A$ and removing $A$ from the label of $v$. Because $S$ has multiplicity $1$ in $\mathfrak S$, $v$ cannot be a leaf of $\mathcal T$ satisfying $\lambda(v) = A$, so the labeled tree $\mathcal T'$ obtained this way is a $\mathcal M$-tree. By construction, $\mathcal T'$ is a 
representation of $\mathfrak S$, so $\mathfrak S$ is compatible.
\end{proof}

For a split system $\mathfrak S=\{S_1, \ldots, S_n\}$, $n \geq 1$ on a multiset $\mathcal M$, we say that a set $A \subseteq M$ is a \emph{terminal set} of $\mathfrak S$
if $A|\overline A=S_i$ for some $i \in \{1, \ldots, n\}$, and the indegree of the vertex $(A,S_i)$ in $\Gamma(\mathfrak S)$ is $0$. Equivalently, $A$ is a terminal set of $\mathfrak S$ if for 
all splits $S_j=B|\overline B$ , with $j \in \{1, \dots, n\}$ distinct from $i$, neither 
$B \subsetneq A$ nor $\overline B \subsetneq A$ holds. We denote by $\mathfrak S_t$ 
the (multi)set of terminal sets of $\mathfrak S$. Clearly, if $\mathfrak S$ is nonempty, then $\mathfrak S_t$ is 
nonempty. For example, for the split system 
$\mathfrak S=\{ab|\overline{ab}, ac|\overline{ac}, cx|\overline{cx}, abcx|abcy\}$ 
represented by the tree in Figure~\ref{fig-intro}, we have $\mathfrak S_t=\{ab,ac,cx\}$. 
Terminal sets have a special place in representations of compatible split systems, 
as the following, straightforward to prove result shows:

\begin{lemma}\label{lmleaves}
Let $\mathfrak S$ be a split system on $\mathcal M$.  If $\mathfrak S$ is 
compatible and $\mathcal T=(T,\lambda)$ is a representation of $\mathfrak S$, 
then for all $A \in \mathfrak S_t$, there exists a leaf $x$ of $T$ such that $\lambda(x)=A$.
\end{lemma}

Note that if $\mathfrak S$ is compatible with representation $\mathcal T=(T,\lambda)$, the 
injection $\mathfrak S_t \to L(T)$ given by Lemma~\ref{lmleaves} is not 
necessarily a bijection. For example, the split system $\mathfrak S=\{ab|abcd,abc|abd\}$ is 
compatible and $\mathfrak S_t=\{ab\}$, but any
representation $\mathcal T=(T,\lambda)$ of $\mathfrak S$ has two leaves, so 
one leaf of $T$ has a label that is not in $\mathfrak S_t$.

Motivated by this fact, for a split system $\mathfrak S=\{S_1, \ldots, S_n\}$, $n \geq 1$, on a multiset $\mathcal M$,
we call a terminal set $A$ of $\mathfrak S$ a \emph{superterminal set} of $\mathfrak S$ if 
$indegree(A,S_i)=0$ and $outdegree(A,S_i)= |\mathfrak S|-1$ in $\Gamma(\mathfrak S)$. 
Equivalently, $A$ is a superterminal set of $\mathfrak S$ 
if and only if for all splits $S_j=B|\overline B$, $j \neq i \in \{1,\dots,n\}$, of $\mathfrak S$, neither 
$B \subsetneq A$ nor $\overline B \subsetneq A$ hold, and exactly 
one of $A \subsetneq B$ or $A \subsetneq \overline B$ holds.
For example, the split system $\mathfrak S=\{ab|cxabcy, ac|bxabcy, cx|ababcy, abcx|abcy\}$ 
has precisely one superterminal set, that is $cx$.
We denote by $\mathfrak S_t^{\times} \subseteq \mathfrak S_t$ the (multi)set of all superterminal sets of $\mathfrak S$. Moreover, for $A \in \mathfrak S_t^{\times}$ 
and $S=B|\overline B \in \mathfrak S$ distinct from $A|\overline A$ we denote by $S(A)$ 
the unique element in $\{B,\overline B\}$ satisfying $A \subseteq S(A)$. To extend 
this notation to $\mathfrak S$, we also put $S(A)=A$ for $S=A|\overline A$. 

Note that if $\mathfrak S$ is thin, then 
$\mathfrak S_t^{\times} = \mathfrak S_t \neq \emptyset$. 
Although thin split systems are pairwise compatible, the converse does not necesarily holds. The following, straightforward to prove  result helps identify certain superterminal sets in 
a pairwise compatible split system.

\begin{lemma}\label{lmsupcore}
Let $\mathfrak S$ be a split system on $\mathcal M$ such that the splits of $\mathfrak S$ 
are pairwise compatible, and let $A \in \mathfrak S_t$. If there exists $a \in \mathcal M$ 
such that $a \notin \overline A$, then $A \in \mathfrak S_t^{\times}$. In particular,  
if $A^* \neq \emptyset$, then $A \in \mathfrak S_t^{\times}$.
\end{lemma}

Superterminal sets play a key role in Proposition~\ref{prdel}, which 
is central to proving the main result of this section.

\begin{proposition}\label{prdel}
Let $\mathfrak S$ be an incompatible split system on $\mathcal M$ of size $k \geq 3$
such that all subsets of $\mathfrak S$ are compatible. If 
$\mathfrak S_t^{\times} \neq \emptyset$, then $\Delta(\mathcal M) \geq k-2$. 
\end{proposition}

\begin{proof}
We prove the proposition by induction on $k$. The case $k=3$ is trivial; for three splits $S_1, S_2, S_3$ on a set $X$, if all three splits are pairwise 
compatible, then the set $\{S_1,S_2,S_3\}$ is compatible. So if three splits are pairwise 
compatible but their union is not, these splits must be defined on a (proper) multiset 
$\mathcal M$, that is, satisfying $\Delta(\mathcal M) \geq 1=k-2$.

Now, assume that the proposition is true for all $k$ lower than a certain threshold $n$, and 
let $\mathfrak S$ be an incompatible split system on $\mathcal M$ of size $n$ such 
that $\mathfrak S_t^{\times} \neq \emptyset$ and all subsets of $\mathfrak S$ are compatible.

Let $A_0 \in \mathfrak S_t^{\times}$, and let $S_0=A_0|\overline{A_0}$. We first
show that $S_0$ must have multiplicity $1$ in $\mathfrak S$. To see this, assume 
by contradiction that this is not the case, and let $\mathcal T=(T,\lambda)$ be a 
representation of $\mathfrak S-\{S_0\}$. Since $S_0 \in \mathfrak S-\{S_0\}$, there 
exists a leaf $x_0$ of $\mathcal T$ with $\lambda(x_0)=A_0$. Moreover, because 
$A_0 \in \mathfrak S_t^{\times}$, the vertex $v$ adjacent to $x_0$ satisfies $A_0 \subseteq \lambda(v)$. 
Thus, taking $A_0$ out of $v$ leads to a representation of $\mathfrak S$, which is impossible 
as $\mathfrak S$ is not compatible by assumption.

Now, let $a_0 \notin \mathcal M$ be a new element. We define the 
split system $\mathfrak S^-$ on $\mathcal M^-=\mathcal M-A_0 \cup \{a_0\}$ as follows:
\[
\mathfrak S^-=\{(S(A_0)-A_0 \cup \{a_0\})|B, S=S(A_0)|B \in \mathfrak S-\{S_0\}\}
\]

Roughly speaking, $\mathfrak S^-$ is obtained by merging, in all splits 
$S \in \mathfrak S-\{S_0\}$, the elements of $A_0$ in $S(A_0)$ into a single element $a_0$. 
By definition, the set $\mathfrak S^-$ is a set of $n-1$ splits on $\mathcal M^-$. Since $S_0$ 
has multiplicity $1$ in $\mathfrak S$, we also have that the split $a_0|(\mathcal M^- - \{a_0\})$ 
is not a split of $\mathfrak S^-$. Moreover, by Lemma~\ref{lmcol}, the set $A_0$ contains at 
least one element of multiplicity two or more in $\mathcal M$, so $\Delta(\mathcal M^-) < \Delta(\mathcal M)$.  
We claim that all subsets of $\mathfrak S^-$ of size $n-2$ are compatible, and $\mathfrak S^-$ is not 
compatible. If the claim is true, it follows by our induction assumption that $\Delta(\mathcal M^-) \geq n-3$. 
Since $\Delta(\mathcal M^-) < \Delta(\mathcal M)$ by construction, it follows that $\Delta(\mathcal M) \geq n-2$.

To show that all subsets of $\mathfrak S^-$ of size $n-2$ are compatible, 
let $S^- \in \mathfrak S^-$ and let $S$ be the corresponding split in $\mathfrak S$ ($S$ is 
the split obtained from $S^-$ by ``replacing" $a_0$ with $A_0$). Since $\mathfrak S-\{S\}$ 
is compatible by assumption, there is a representation $\mathcal T=(T,\lambda)$ of $\mathfrak S-\{S\}$. 
Moreover, by Lemma~\ref{lmleaves}, there exists a leaf $x$ of $T$ with $\lambda(x)=A_0$. Thus 
it is not difficult to see that the tree $T'$ obtained from $T$ by changing the label 
of $x$ from $A_0$ to $a_0$, and then collapsing the edge incident to $x$, is a 
representation of $\mathfrak S^--\{S^-\}$.

To show that $\mathfrak S^-$ is not compatible, assume by contradiction that this is 
not the case, and let $\mathcal T=(T,\lambda)$ be a representation of $\mathfrak S^-$. 
Then there must exist a vertex $v_0$ of $T$ such that $a_0 \in \lambda(v_0)$. As mentioned 
before, the split $a_0|(\mathcal M^+ - \{a_0\})$ does not belong to $\mathfrak S^-$, so if $v_0$ 
is a leaf, $\lambda(v_0)-\{a_0\}$ is nonempty. By removing $a_0$ from the label of $v_0$ and 
adding a leaf $x$ adjacent to $v_0$ labeled with $A_0$, we then obtain a representation of $\mathfrak S$. 
However, this is impossible, as $\mathfrak S$ is not compatible by assumption. So, $\mathfrak S^-$ is not compatible.
\end{proof}

We now prove the main theorem of this section.

\begin{theorem}\label{tmd2}
Let $\mathfrak S$ be a split system on a multiset $\mathcal M$, such that every 
subset $\mathfrak S'$ of $\mathfrak S$ of size at least  $\Delta(\mathcal M)+3$ 
satisfies $(\mathfrak S')_t^{\times} \neq \emptyset$. Then $\mathfrak S$ is 
compatible if and only if every subset of $\mathfrak S$ of size at most $\Delta(\mathcal M)+2$ is compatible.
\end{theorem}

\begin{proof}
One direction is trivial. To see the other direction, assume by contradiction that $\mathfrak S$ is 
not compatible. Let $\mathfrak S'$ be an incompatible subset of $\mathfrak S$ that 
is minimal for inclusion, that is, all proper subsets of $\mathfrak S'$ are compatible. 

Denoting by $k$ the size of $\mathfrak S'$, we have by assumption that $k > \Delta(\mathcal M)+2$.
Moreover, because $\mathfrak S'$ has size greater than $\Delta(\mathcal M)+2$, 
we also have by assumption that $(\mathfrak S')_t^{\times} \neq \emptyset$. 
Thus, it follows from Proposition~\ref{prdel} that $\Delta(\mathcal M) \geq k-2$, which is 
impossible. Hence $\mathfrak S$ must be compatible.
\end{proof}

\begin{open}
	Can the condition 
	$\mathfrak S_t^{\times} \neq \emptyset$ 
	be removed from the 
	statement of Proposition~\ref{prdel}?
	If so, then it follows by the proof of Theorem~\ref{tmd2} 
	that Conjecture~\ref{mainconj} would hold.
\end{open}

Interestingly, we immediately obtain as a corollary of Theorem~\ref{tmd2}
that Conjecture~\ref{mainconj} holds for thin split systems, that is:

\begin{corollary}\label{cor-thin}
	Suppose that $\mathfrak S$ is a thin split system on a multiset $\mathcal M$.
	Then $\mathfrak S$ is compatible if and only if every subset 
	of $\mathfrak S$ of size at most $\Delta(\mathcal M)+2$ is compatible.
\end{corollary}

\begin{proof}
It suffices to remark that if $\mathfrak S$ is thin, then 
all subsets of $\mathfrak S$ are thin. As noted before, thin split 
systems satisfy $\mathfrak S_t^{\times} = \mathfrak S_t \neq \emptyset$. 
Hence, all subsets $\mathfrak S'$ of $\mathfrak S$ 
satisfy $(\mathfrak S')_t^{\times} \neq \emptyset$, so Theorem~\ref{tmd2} 
applies.
\end{proof}

Note that the bound given in the last corollary is tight, since 
there exist thin split systems $\mathfrak S$ such that
every subset of size at most $\Delta(\mathcal M)+1$ is compatible 
but $\mathfrak S$ is not compatible. For example, let $n>m \geq 1$ 
and consider the split system 
$\mathfrak S=\{a_{i,1} \ldots a_{i,m}x|\overline{a_{i,1} \ldots a_{i,m}x},i=\{1, \ldots, n\}\}$ 
on $\mathcal M=\{a_{i,j}: 1 \leq i \leq n, 1 \leq j \leq m\} \cup \{x^{n-1}\}$, 
where all $a_{i,j}$ are pairwise distinct and distinct from $x$. Clearly, $\mathfrak S$ 
is tight, and we have $|\mathfrak S|=n$. It is easy to see that $\mathfrak S$ is 
not compatible, but all proper subsets of $\mathfrak S$ are compatible. 
Since $\Delta(\mathcal M)=n-2$, we have in particular that all 
subsets of $\mathfrak S$ of size $\Delta(\mathcal M)+1$ are compatible.


\section{Splits of restricted size}\label{sec-23}

Suppose that $S=A|\overline A$ is a split on $\mathcal M$. 
We define the \emph{size} of $S$ as $\mathrm{min}\{|A|,|\overline A|\}$. 
In this section, in order to gain a deeper understanding for Conjecture~\ref{mainconj}, 
we focus on understanding compatibility of 
split systems in which the splits have some restriction on their size. 

We first note that as a consequence of Corollary~\ref{cor-thin} we immediately 
obtain the following generalization of \cite[Lemma 4.6 (ii)]{HLMS08}.

\begin{theorem}\label{equal}
	Suppose that $\mathfrak S$ is a split system on a multiset $\mathcal M$ in which every split has the 
	same size. Then $\mathfrak S$ is compatible if and only if every
	submultiset of $\mathfrak S$ of size at most $\Delta(\mathcal M)+2$ is compatible.
\end{theorem}

\begin{proof}
	Suppose that every
	submultiset of $\mathfrak S$ of size at most $\Delta(\mathcal M)+2$ is compatible.
	Then every pair of splits in $\mathfrak S$ is compatible.
	Since every split in  $\mathfrak S$  has the same size it follows that $\mathfrak S$ is thin. 
	Now we can apply Corollary~\ref{cor-thin} to see that $\mathfrak S$ is compatible.
\end{proof}

For the rest of this section we consider splits systems in which
every split has size 2 or 3, which we shall call 
{\em 2,3-split systems}. In particular we will prove in Theorem~\ref{twothreeholds} that such split systems, if they only contain splits of multiplicity 1, enjoy a slightly weaker version of ($\bigstar$).

We first prove a useful technical lemma.

\begin{lemma}\label{inc-tree}
Let $\mathcal M$ be a multiset with underlying set $X$ and 
let $A_1, \ldots, A_k$, 
$k \geq 2$ be a partition of $\mathcal M$. Let $G$ be a graph whose vertex set is the set $\{A_i, 1\leq i \leq k\}$, 
such that each edge $\{A_i,A_j\}$, $i \neq j$ of $G$ satisfies 
$A_i \cap A_j \neq \emptyset$. Then, we have $\Delta(\mathcal M) \geq k-c$, where 
$c$ is the number of connected components of $G$.
\end{lemma}

\begin{proof}
Let $F$ be the union of some spanning trees of each components of $G$. 
Since $F$ is acyclic, we have $|E(F)|=|V(F)|-c=k-c$, so it suffices to show that $\Delta(\mathcal M) \geq |E(F)|$.

To each edge $e=\{A,B\}$ of $F$, we can 
associate one element $p(e)$ in $A \cap B$. For $x$ an 
element of $X$, we denote by $\pi(x)$ the number of edges $e \in E(F)$ 
such that $p(e)=x$. Since $F$ is acyclic, the number of vertices $A$ of $F$ 
satisfying $x \in A$ is at least $\pi(x)+1$. Since the vertices of $F$ form 
a partition of $\mathcal M$, it follows that $m(x) \geq \pi(x)+1$. This 
implies $\Delta(\mathcal M) = \Sigma_{x \in X} (m(x)-1) \geq \Sigma_{x \in X} \pi(x)=|E(F)|$, as required.
\end{proof}

Using Lemma~\ref{inc-tree}, we now prove a key proposition.

\begin{proposition}\label{twothree}
	Suppose $\mathfrak F$ is an incompatible 2,3-split system on $\mathcal M$ containing $k \geq 3$ splits,
	in which every split has multiplicity 1
	such that all subsets of $\mathfrak F$ are compatible. Then $\Delta(\mathcal M) \geq k-3$. 
\end{proposition}
\begin{proof}
If $\mathfrak F_t^{\times} \neq \emptyset$, then $\Delta(\mathcal M) \geq k-2 \ge k-3$ 
holds by Proposition~\ref{prdel}. So, assume $\mathfrak F_t^{\times}= \emptyset$ holds. 
In particular, $\mathfrak F$  contains at least one 2-split and 
at least one 3-split. We begin by showing that $\mathfrak F$ enjoys the following property:

 \begin{itemize}
 	\item[(*)] If $A|\overline A \in \mathfrak F$ such that $|A|=3$, then there exists some $B|\overline B, C|\overline C \in \mathfrak F$ with $B\neq C$, $|B|=|C|=2$ and $B,C \subseteq A$.
 \end{itemize}
 
 To see this, first note that there exists some $B|\overline B \in \mathfrak F$ with $|B|=2$ 
 and $B \subset A$, otherwise $A \in \mathfrak F_t^{\times}$. Now, suppose
 there does not exist some  $C|\overline C$ in addition to $B|\overline B$. 
 Then since $\mathfrak F - \{B|\overline B\}$ is	compatible, there is some 
 tree representing $\mathfrak F - \{B|\overline B\}$ with a leaf 
 having label set $A$. But then $\mathfrak F$ is compatible, a contradiction.

Now, let $S^*$ be a 2-split of $\mathfrak F$ such that $A(S^*)=\{x,y\}$. Let $\mathfrak S = \mathfrak F - S^*$.
Since $\mathfrak S$ is compatible, there is some 
$\mathcal M$-tree $\mathcal T = (T,\lambda)$ which represents $\mathfrak S$.
Note that as $\mathfrak S$ has size at least 4, $T$ has a ``central vertex" $v^*$ 
that is adjacent to all edges that correspond to 3-splits of $\mathfrak S$, and 
all other vertices in $T$ have degree 1 or 2.
Let $\mathcal M'=\mathcal M-\lambda(v^*)$. Since $\mathcal M' \subseteq \mathcal M$, it 
follows that $\Delta(\mathcal M') \leq \Delta(\mathcal M)$. We next proceed to show that $\Delta(\mathcal M') \geq k-3$.

Denote by $V_0$ (\emph{resp.} $V_1$) the set of leaves of $\mathcal T$ (\emph{resp.} 
the set of vertices of degree 2 of $\mathcal T$, excluding $v^*$). We also 
denote by $\lambda(V_0)$ (\emph{resp.} $\lambda(V_1)$) the multiset union 
of the multisets $\lambda(v)$, $v \in V_0$ (\emph{resp.} $v \in V_1$). These 
sets form a partition $\lambda(V_0) \cup \lambda(V_1)$ of $\mathcal M'$. 
Note also that for all $v \in V_1$, we have $|\lambda(v)|=1$, and that $|V_0|+|V_1|=|\mathfrak S|$.

Now, let $v \in V_1$ and let $z_v$ be the unique element in $\lambda(v)$. 
If $z \notin \{x,y\}$, then by (*), there exists a leaf $l$ of $\mathcal T$ 
such that $z_v \in \lambda(l)$. Otherwise, if $z_v \in \{x,y\}$, say $z_v=x$, 
then by (*), either there exists a leaf $l$ such that $x \in \lambda(l)$, or 
the leaf $l_v$ adjacent to $v$ satisfies $y \in \lambda(l_v)$ (note that  
these two cases are not mutually exclusive). Put together, these observation 
imply that there is at most one element in $\mathcal M$ (which must be either 
$x$ or $y$) that belongs to the label of some vertices in $V_1$ but does not 
belong to the label of any vertex of $V_0$. This means 
that $\Delta(\mathcal M') = \Delta(\lambda(V_0))+|V_1|$ if there exists $l_x$ 
and $l_y$ in $V_0$ (necessarily distinct) such that $x \in \lambda(l_x)$ 
and $y \in \lambda(l_y)$, and $\Delta(\mathcal M') = \Delta(\lambda(V_0))+|V_1|-1$ otherwise.

We next focus our attention on the elements in the set $V_0$. 
We define the undirected graph $G(V_0)$ to be 
the graph with vertex set $\{\lambda(v), v \in V_0\}$,
in which two distinct sets $\lambda(u),\lambda(v), u,v \in V_0$ are 
joined by an edge if $\lambda(u) \cap \lambda(v) \neq \emptyset$. 
We claim that $G(V_0)$ has at most two connected components, and 
has two connected components only if there exists $l_x$ and $l_y$ in $V_0$ (necessarily distinct) 
such that $x \in \lambda(l_x)$ and $y \in \lambda(l_y)$.

To prove this claim, we
 first remark that since all vertices $v$ of $V_0$ are leaves of $\mathcal T$, 
 we have $\lambda(v_0)|\overline{\lambda(v_0)} \in \mathfrak S$, where $\lambda(v_0)$ is 
the smallest part of $\lambda(v_0)|\overline{\lambda(v_0)}$. Hence, $G(V_0)$ 
is a subgraph of the graph $G(\mathfrak S)$ with 
vertex set $\{A \subset \mathcal M: A \text{ is the small part of a split } S \in \mathfrak S\}$, 
and in which two sets $A$ and $B$ are joined by an edge if $A \cap B \neq \emptyset$. 
More precisely, $G(V_0)$ is obtained from $G(\mathfrak S)$ by 
removing the set $\{x,y\}$ from $V(G(\mathfrak S))$, and all sets of size 
three that do not label leaves of $\mathcal T$.

We next show that $G(\mathfrak S)$ is connected. Assume for contradiction that this is not the case. 
Then there exists a partition $\mathfrak S_1, \mathfrak S_2$ of $\mathfrak S$ 
such that for any two splits $S_1 \in \mathfrak S_1$, $S_2 \in \mathfrak S_2$, 
the smallest parts of $S_1$ and $S_2$ do not intersect. 
Since $\mathcal S_1$ and $\mathcal S_2$ are nonempty 
proper subsets of $\mathcal S$, it follows that 
there exist two $\mathcal M$-trees $\mathcal T_1=(T_1,\lambda_1)$, $\mathcal T_2=(T_2,\lambda_2)$ 
representing $\mathfrak S_1$ and $\mathfrak S_2$, respectively. 
Moreover, there exist two vertices $v_1 \in V(T_1)$, $v_2 \in V(T_2)$ 
such that for $i,j \in \{1,2\}$ distinct, 
$\bigcup_{v \in V(T_i)-\{v_i\}} \lambda_i(v) \subseteq \lambda_j(v_j)$. 
By identifying vertices $v_1$ and $v_2$, and 
defining the label of the newly created vertex as $\lambda_1(v_1) \cap \lambda_2(v_2)$, 
we obtain a representation of $\mathfrak S$, which is 
impossible since $\mathfrak S$ is not compatible. Hence, $G(\mathfrak S)$ is connected.

To complete the proof of the claim, we now further consider the 
relationship between $G(V_0)$ and $G(\mathfrak S)$. 
First, let $G^-$ be the graph obtained from $G(\mathfrak S)$ by 
removing all sets of size three that are not of the 
form $\lambda(v_0), v_0 \in V_0$. In view of (*), removing 
sets of size three from the vertex set of $G(\mathfrak S)$ 
does not modify the number of connected components 
of the resulting graph, so $G^-$ is connected. Recall 
that $G(V_0)$ is obtained from $G^-$ by removing 
the vertex $\{x,y\}$. In particular $V_0=V(G^-)-\{\{x,y\}\}$. 
By definition, each set adjacent to $\{x,y\}$ in $G^-$ contains 
exactly one of $x$ or $y$. Note that no vertex in $G(V_0)$ 
contains both $x$ and $y$, as otherwise there 
exists a leaf $l$ of $\mathcal T$ with $\{x,y\} \subsetneq \lambda(l)$, 
which is impossible since $\mathfrak S$ is not compatible. 
Since all vertices of $G(V_0)$ containing $x$ (\emph{resp.} 
all vertices of $G(V_0)$ containing $y$) form a clique in $G(V_0)$, it
follows that $G(V_0)$ has at most two connected components. 
Moreover, if $G(V_0)$ has two connected components, then 
there exists at least one vertex adjacent to $\{x,y\}$ containing $x$, 
and one vertex adjacent to $\{x,y\}$ in $G^-$ containing $y$. 
This means that there exists two leaves $l_x$, $l_y$ of $\mathcal T$ such 
that $\lambda(l_x)$ contains $x$ and $\lambda(l_y)$ contains $y$, which 
completes the proof of the claim.

To conclude the proof of the theorem 
we distinguish between the cases where $G(V_0)$ has one 
and two connected components.

If $G(V_0)$ has one connected component, then it follows by Lemma~\ref{inc-tree} 
that $\Delta(\lambda(V_0)) \geq |V_0|-1$. Since as stated above, 
we have $\Delta(\mathcal M') \geq \Delta(\lambda(V_0))+|V_1|-1$, it 
follows that $\Delta(\mathcal M') \geq |V_0|+|V_1|-2=|\mathfrak S|-2=k-3$, 
which concludes the proof of the theorem. 

If $G(V_0)$ has two connected components, then it follows by Lemma~\ref{inc-tree}
that $\Delta(\lambda(V_0)) \geq = |V_0|-2$. Moreover, the fact that $G(V_0)$ 
has two connected components also implies that there exists $l_x$ and $l_y$ in $V_0$ 
such that $x \in \lambda(l_x)$ and $y \in \lambda(l_y)$. As stated above, 
this in turn implies that $\Delta(\mathcal M') = \Delta(\lambda(V_0))+|V_1|$. 
Putting these two relationships together, we get 
$\Delta(\mathcal M') \geq |V_0|+|V_1|-2=|\mathfrak S|-2=k-3$, which concludes the 
proof of the theorem.
\end{proof}

We now prove the main result of this section.

\begin{theorem}\label{twothreeholds}
	Suppose that $\mathfrak S$ is a 2,3-split system on a multiset $\mathcal M$ in which every split has multiplicity 1. 
	Then $\mathfrak S$ is compatible if and only if every
	submultiset of $\mathfrak S$ of size at most $\Delta(\mathcal M)+3$ is compatible.
\end{theorem}

\begin{proof}
The `only if' direction is trivial. To see the `if' direction, assume for contradiction that $\mathfrak S$ is 
not compatible. Let $\mathfrak S'$ be an incompatible subset of $\mathfrak S$ that 
is minimal for inclusion, that is, all proper subsets of $\mathfrak S'$ are compatible.
Denoting by $k$ the size of $\mathfrak S'$, we have by assumption that $k > \Delta(\mathcal M)+3$. 
On the other hand, Proposition~\ref{twothree} applied 
to $\mathfrak S'$ implies that $\Delta(\mathcal M) \geq k-3$, a contradiction.
\end{proof}

\begin{remark}
There are examples of 2,3-split systems $\mathfrak F$ 
that satisfy the conditions of Proposition~\ref{twothree} 
and such that $\mathfrak F_t^{\times} = \emptyset$ 
(and so Proposition~\ref{prdel} does not apply). 
For example, take
$$
\mathfrak F=\{ab|\overline{ab},ac|\overline{ac},bc|\overline{bc},
cd|\overline{cd},ce|\overline{ce},de|\overline{de},abc|\overline{abc},cde|\overline{cde}\}
$$
on $\mathcal M=\{a^2,b^2,c^5,d^2,e^2\}$. Note that in this
example, $8=\Delta(\mathcal M)>|\mathfrak F|-3=5$.
\end{remark}

\subsection*{Acknowledgements}

GES would like to thank the streets of Norwich and the seaside town 
of Sheringham for inspirational autumn walks, during which ideas for some of 
the proofs arose. Both authors thank Katharina Huber for helpful
discussions on compatibility.


\bibliographystyle{plain}
\bibliography{bibliography}

\begin{thebibliography}{1}

\bibitem{B71}
P.~Buneman.
\newblock The recovery of trees from measures of dissimilarity.
\newblock {\em Mathematics in the Archeological and Historical Sciences}, pages
  387--395, 1971.

\bibitem{G06}
G.~Ganapathy, B.~Goodson, R.~Jansen, H.-S. Le, V.~Ramachandran, and T.~Warnow.
\newblock Pattern identification in biogeography.
\newblock {\em IEEE Trans. Comput. Biol. Bioinform.}, 3:334--346, 2006.

\bibitem{HLMS08}
K.~T. Huber, M.~Lott, V.~Moulton, and A.~Spillner.
\newblock The complexity of deriving multi-labeled trees from bipartitions.
\newblock {\em Journal of Computational Biology}, 15:639--651, 2008.

\bibitem{huber2006phylogenetic}
Katharina~T Huber and Vincent Moulton.
\newblock Phylogenetic networks from multi-labelled trees.
\newblock {\em Journal of mathematical biology}, 52(5):613--632, 2006.

\bibitem{P94}
R.~D.~M. Page.
\newblock Maps between trees and cladistic analysis of historical associations
  among genes, organisms, and areas.
\newblock {\em Systematic Biology}, 43:58--77, 1994.

\bibitem{PO01}
M.~Popp and B.~Oxelman.
\newblock Inferring the history of the polyploid silene aegaea
  (caryophyllaceae) using plastid and homoeologous nuclear dna sequences.
\newblock {\em Molecular Phylogenetics and Evolution}, 20:474--481, 2001.

\bibitem{SS03}
C.~Semple and M.~Steel.
\newblock {\em Phylogenetics}.
\newblock Oxford University Press, 2003.

\end{thebibliography}

\end{document}